\documentclass[a4paper,10pt]{amsart}

\usepackage{amssymb,amsmath,amsfonts}

\title[A model for the orbifold Chow ring of $\IP(w)$]{A model
  for the orbifold Chow ring \\of weighted projective spaces}
\author{Samuel Boissi\`ere \and \'Etienne Mann \and Fabio Perroni}
\thanks{The third author was partially supported by SNF grant No 200020-107464/1}
\address{Samuel Boissi{\`e}re, Laboratoire J.A.Dieudonn\'e, UMR CNRS 6621,
         Universite de Nice Sophia-Antipolis, Parc Valrose, 06108 Nice}
\email{sb@math.unice.fr}
\address{\'Etienne Mann, SISSA, Via Beirut 2-4, 34014 Trieste, Italy}
\email{mann@sissa.it}
\address{Fabio Perroni, Institut f\"ur Mathematik, Universit\"at Z\"urich, Winterthurerstrasse 190,
8057 Z\"urich, Switzerland}
\email{fabio.perroni@math.unizh.ch}

\def\cf{\textit{cf.}\kern.3em}

\def\resp{\textit{resp.}\kern.3em}

\numberwithin{equation}{section} \makeatletter

\newcommand{\IN}{\mathbb{N}}

\newcommand{\IQ}{\mathbb{Q}}
\newcommand{\IC}{\mathbb{C}}

\newcommand{\IP}{\mathbb{P}}

\newcommand{\ii}{\mathrm{i}}
\newcommand{\jj}{\mathrm{j}}

\newcommand{\ds}{\displaystyle}

\newcommand{\bs}{\boldsymbol}

\newcommand{\cO}{\mathcal O}
\newcommand{\cU}{\mathcal U}

\newcommand{\model}[2]{\ll #1,#2\gg}

\newcommand{\kmin}{k_{\text{min}}}
\newcommand{\kmax}{k_{\text{max}}}

\DeclareMathOperator{\gr}{\mathrm gr} \DeclareMathOperator{\orb}{\mathrm orb}
  
\DeclareMathOperator{\Sp}{\mathrm Sp}

\theoremstyle{plain}
\newtheorem{theorem}[equation]{Theorem}
\newtheorem{proposition}[equation]{Proposition}
\newtheorem{lemma}[equation]{Lemma}

\theoremstyle{definition}
\newtheorem{example}[equation]{Example}
\newtheorem{remark}[equation]{Remark}

\begin{document}

\begin{abstract}
We construct an isomorphism of graded Frobenius algebras between the orbifold
Chow ring of weighted projective spaces and graded algebras of groups
of roots of the unity.
\end{abstract}

\maketitle

\section{Introduction}

Recently, particular attention has been given to the study of the orbifold
cohomology ring of weighted projective spaces. This cohomology originates in
physics and has been defined mathematically by Chen and Ruan in \cite{CR}. It
has been further developed and adapted to the language of stacks by Abramovich,
Graber and Vistoli in \cite{AGV}. In \cite{Jiang} Jiang studies the orbifold
cohomology ring of weighted projective spaces by means of their simplicial
toric fan with an explicit computation for $\IP(1,2,2,3,3,3)$. In \cite{BCS}
Borisov, Chen and Smith prove a formula for the orbifold Chow ring of toric
Deligne-Mumford stacks in terms of their stacky fan, that can be applied to
weighted projective spaces. In \cite{CH}, Chen and Hu obtain a general formula
for the computation of the orbifold product of abelian orbifolds, and apply it
to weighted projective spaces. In \cite{Holm}, Holm uses symplectic geometry to
compute a presentation of their integral orbifold cohomology ring with
generators and relations. In this note, we give an alternative description:
Starting from the computation of their orbifold cohomology as given by the
second author in \cite{Mann}, we exhibit a short and comprehensible model as a
graded group algebra over some specific group of roots of the unity.

More precisely, for a given sequence of weights $w:=(w_0,\ldots,w_n)$, we
construct (see Theorem \ref{th:isom}) an isomorphism of graded Frobenius
algebras between the orbifold Chow ring $A^\star_{\orb}(\IP(w))$ and a suitably graded
algebra $\gr_F^\star\IC[\cU_{|w|}]$ where $|w|:=w_0+\cdots+w_n$ and $\cU_{|w|}$
denotes the group of $|w|$-roots of the unity (see \S\ref{subsec:model} for the
construction of the grading). This construction uses some combinatorics
associated to the weights, partially present in Douai and Sabbah \cite{DS} and
the second author in \cite{Mann} (see \S\ref{section:combinatorics}).

This new description of the orbifold Chow ring of weighted projective spaces is
interesting mainly for three reasons. First, this model is not a description
with generators of the ring and a list of relations: Instead, it gives a
presentation such that the generators of the ring are a basis of the underlying
vector space, and the ring structure is natural. Second, this description is
analogous to the one of a global symplectic quotient stack $[V/G]$ where $V$ is
a symplectic vector space and $G\subset\Sp(V)$ a finite group: Ginzburg and
Kaledin in \cite{GK} observed that the ring $A^\star_{\orb}([V/G])$ is
isomorphic to $\gr_F^\star\IC[G]^G$ (for an appropriate grading). Third, the
computation of this orbifold Chow ring is an essential step for studying the
Cohomological Crepant Resolution Conjecture for weighted projective spaces (see
\cite{BMP} for such computations).

\section{Some combinatorics}
\label{section:combinatorics}

Let $n\geq 1$ be an integer and $w:=(w_0,\ldots,w_n)$ a sequence of positive
integers (or \emph{weights}). Set $|w|:=w_0+\cdots+w_n$. For $\ell\in\IN^*$,
denote by $\cU_\ell$ the group of $\ell$-th roots of the unity and set
$\cU:=\bigcup\limits_{\ell\in\IN^*}\cU_\ell$. We define an order on the group
$\cU$ by taking the principal determination of the argument, inducing a
bijection:
$$
\begin{array}{rcl}\gamma\colon\,\cU &\longrightarrow& [0,1[\cap\IQ\\
 g&\longmapsto &\gamma(g) \text{ where } \exp(2\ii\pi\gamma(g))=g.
\end{array}
$$
For $g\in\cU$, put:
\begin{align*}
I(g)&:=\left\{i\in[\![0,n]\!]\mid g^{w_i}=1\right\},\\
a(g)&:=\sum_{i=0}^n\{\gamma(g)w_i\},
\end{align*}
where $\{\cdot\}$ is the fractional part. Note that $I(g)=I(g^{-1})$. One has:
\begin{equation}
\label{eq:1} \{\gamma(g)w_i\}=\begin{cases} 0 &\text{if } i\in I(g)\\
1-\{\gamma(g^{-1})w_i\}& \text{otherwise}\end{cases}
\end{equation}
hence:
\begin{equation}
\label{eq:2} a(g)+a(g^{-1})=n+1-\# I(g).
\end{equation}
We order the disjoint union $\bigsqcup\limits_{i=0}^n\cU_{w_i}$ by the injection:
$$
\begin{array}{rcl}
\ds\bigsqcup\limits_{i=0}^n\cU_{w_i}&\longrightarrow &[0,1[\times
[\![0,n]\!]\\
g&\longmapsto&(\gamma(g),i),
\end{array}
$$
where $[0,1[\times [\![0,n]\!]$ is given the lexicographic order. This induces
an increasing bijection
$s\colon[\![0,|w|-1]\!]\to\bigsqcup\limits_{i=0}^{n}\cU_{w_{i}}$.

\begin{example}
Take $w=(1,2,3)$. The enumeration is:
\begin{align*}
\gamma&\colon\cU_1\sqcup\cU_2\sqcup\cU_3\to\{0\}\sqcup\{0,1/2\}\sqcup\{0,1/3,2/3\},\\
\gamma s&\colon[\![0,5]\!] \to [0,0,0,1/3,1/2,2/3].
\end{align*}
\end{example}

The understanding of the growth of the composite map $\gamma
s:[\![0,|w|-1]\!]\to[0,1[$ is central in the sequel. First note that for
$g\in\cU$, the elements in the image of $\gamma s$ less or equal to $\gamma(g)$
are:
\begin{equation}
\label{eq:3}
0,\frac{1}{w_0},\ldots,\frac{[\gamma(g)w_0]}{w_0},0,\frac{1}{w_1},\ldots,\frac{[\gamma(g)w_1]}{w_1},\ldots,
0,\frac{1}{w_n},\ldots,\frac{[\gamma(g)w_n]}{w_n}
\end{equation}
where $[\cdot]$ is the integer part. In particular, $\#(\gamma
s)^{-1}(\gamma(g))=\#I(g)$. The growth is then controlled by the values, for
$g\in\bigcup\limits_{i=0}^n \cU_{w_i}$:
\begin{align*}
\kmin(g)&:=\min\left\{k\in[\![0,|w|-1]\!]\mid\gamma s(k)=\gamma(g)\right\},\\
\kmax(g)&:=\max\left\{k\in[\![0,|w|-1]\!]\mid\gamma s(k)=\gamma(g)\right\}.
\end{align*}
One has the relation: $\kmax(g)=\kmin(g)+(\#I(g)-1)$. Otherwise stated:
\begin{equation}
\label{eq:4} \gamma s(\kmin(g)+d)=\gamma (g)\quad \forall d=0,\ldots,\#I(g)-1.
\end{equation}
Another consequence of (\ref{eq:3}) is:
$$
\#\left\{k\in[\![0,|w|-1]\!]\mid\gamma s(k)\leq
\gamma(g)\right\}=n+1+\sum_{i=0}^n[\gamma(g)w_i].
$$
One deduces:
\begin{align*}
\kmin(g)&=(n+1-\#I(g))+\sum_{i=0}^n[\gamma(g)w_i],\\
\kmax(g)&=n+\sum_{i=0}^n[\gamma(g)w_i].
\end{align*}
Using that $\sum\limits_{i=0}^n[\gamma(g)w_i]=|w|\gamma(g)-a(g)$ and Formula
(\ref{eq:2}) one gets:
\begin{align}
\label{eq:5} \kmin(g)&=a(g^{-1})+|w|\gamma(g),\\
\label{eq:5bis} \kmax(g)&=n+|w|\gamma(g)-a(g).
\end{align}

If $g\in\cU$ but $g\notin\bigcup\limits_{i=0}^n\cU_{w_i}$ then $I(g)=\emptyset$
and:
\begin{align*}
\#\left\{k\in[\![0,|w|-1]\!]\mid\gamma s(k)<
\gamma(g)\right\}&=n+1+\sum_{i=0}^n[\gamma(g)w_i]\\
&=n+1+|w|\gamma(g)-a(g)\\
&=|w|\gamma(g)+a(g^{-1}) \text{ by Formula (\ref{eq:2})}
\end{align*}
so we can extend the definition of $\kmin(g)$ by setting:
\begin{equation}
\label{eq:6} \kmin(g):=a(g^{-1})+|w|\gamma(g)\quad \forall g\in\cU,
\end{equation}
with the property that $k\geq \kmin(g)$ if and only if $\gamma
s(k)\geq\gamma(g)$ (resp. $\gamma s(k)>\gamma(g)$ if
$g\notin\bigcup\limits_{i=0}^n\cU_{w_i}$).

For $g,h\in\bigcup\limits_{i=0}^n\cU_{w_i}$, set:
$$
J(g,h):=\left\{i\in[\![0,n]\!]\mid\{\gamma(g)w_i\}+\{\gamma(h)w_i\}+\{\gamma(gh)^{-1}w_i\}=2\right\}.
$$
Using Formula (\ref{eq:1}) and noting that:
\begin{equation}
\label{eq:7} \{\gamma(gh)w_i\}\equiv\{\gamma(g)w_i\}+\{\gamma(h)w_i\} \mod 1
\end{equation}
one gets the following decomposition in disjoint union:
$$
[\![0,n]\!]=(I(g)\cup I(h))\sqcup\left(I(gh)\setminus (I(g)\cap
I(h))\right)\sqcup J(g,h)\sqcup J(g^{-1},h^{-1})
$$
or more precisely:
\begin{equation}
\label{eq:9} \{\gamma(g)w_i\}+\{\gamma(h)w_i\}-\{\gamma(gh)w_i\}=\begin{cases}
0 & \text{if } i\in I(g)\cup I(h)\\1 & \text{if } i\in I(gh)\setminus (I(g)\cap
I(h))\\0& \text{if } i\in J(g^{-1},h^{-1})\\ 1 & \text{if } i\in
J(g,h).\end{cases}
\end{equation}

This implies:
\begin{equation}\label{eq:10}
a(g)+a(h)-a(gh)=\#\left(I(gh)\setminus(I(g)\cap I(h))\right)+\#J(g,h).
\end{equation}

\section{Orbifold Chow ring of weighted projective spaces}

\subsection{Weighted projective spaces}
\label{subsec:weight-proj-spac}\text{}

Let $w:=(w_0,\ldots,w_n)$ be a sequence of weights. The group $\IC^*$ acts on
$\IC^{n+1}\setminus\{0\}$ by:
$$
\lambda\cdot(x_{0}, \ldots,x_{n}):=(\lambda^{w_{0}}x_{0}, \ldots
,\lambda^{w_{n}}x_{n}).
$$
The \emph{weighted projective space} $\IP(w)$ is defined as the quotient stack
$[(\IC^{n+1}\setminus\{0\})/\IC^*]$. It is a smooth proper Deligne-Mumford
stack whose coarse moduli space, denoted $|\IP(w)|$, is a projective variety of
dimension $n$.

For any subset $I:=\{i_{1}, \ldots ,i_{k}\}\subset \{0, \ldots ,n\}$, set
$w_I:=(w_{i_{1}}, \ldots ,w_{i_{k}})$. There is a natural closed embedding
$\iota_{I}:\IP(w_I)\hookrightarrow\IP(w)$. The weighted projective space
$\IP(w)$ comes with a natural invertible sheaf $\cO_{\IP(w)}(1)$ defined as
follows: For any scheme $X$ and any stack morphism $X\to \IP(w)$ given by a
principal $\IC^*$-bundle $P\to X$ and a $\IC^*$-equivariant morphism $P\to
\IC^{n+1}\setminus\{0\}$, one defines $\cO_{\IP(w)}(1)_{X}$ as the sheaf of
sections of the associated line bundle of $P$. This sheaf is compatible with
the embedding $\iota_{I}$ in the sense that
$\iota^*_I\cO_{\IP(w)}(1)=\cO_{\IP(w_I)}(1)$.

\subsection{Computation of the orbifold Chow ring}
\label{subsec:orbifold-chow}\text{}

We denote by $A^\star_{\orb}(\mathcal{X})$ the orbifold Chow ring with
\emph{complex} coefficients of a Deligne-Mumford stack (or \emph{orbifold})
$\mathcal{X}$. For toric stacks, such as weighted projective spaces (see
\cite{BMP}), it is isomorphic to the even orbifold cohomology. As a vector
space, $A^\star_{\orb}(\mathcal{X})=A^\star(\mathcal{IX})$ where $\mathcal{IX}$
is the \emph{inertia stack} of $\mathcal{X}$.

We recall the results of the second author in \cite{Mann}. The inertia stack of
$\IP(w)$ decomposes as:
$$
\mathcal{I}\IP(w)=\coprod_{g\in\bigcup\limits_{i=0}^n\cU_{w_i}}\IP(w_{I(g)}).
$$
Note that $\dim\IP(w_{I(g)})=\#I(g)-1$.

\begin{example}
Take again $w=(1,2,3)$. The components of the inertia stack of $\IP(1,2,3)$ are
indexed by the roots $1,\jj,-1,\jj^2$ where $\jj$ is the primitive third root
of the unity:
$$
\mathcal{I}\IP(1,2,3)=\IP(1,2,3)\amalg\IP(3)\amalg\IP(2)\amalg\IP(3).
$$
\end{example}

For $g\in\bigcup\limits_{i=0}^n\cU_{w_{i}}$ and
$d\in\{0,\ldots,\dim\IP(w_{I(g)})\}$, define the classes\footnote{The
normalizations factor differs from \cite{Mann}.}:
$$
\eta_{g}^{d}:=\left(\prod_{i=0}^{n}w_{i}^{-\{\gamma(g)w_{i}\}}\right)\cdot
c_{1}(\mathcal{O}_{\IP(w_{I(g)})}(1))^{d}\in A^{d}(|\IP(w_{I(g)})|).
$$

The first result concerns the vector space decomposition of
$A^{\star}_{\orb}(\IP(w))$.

\begin{proposition}\label{prop:vector}\emph{\cite[Proposition 3.9 \& Corollary 3.11]{Mann}}
  \begin{enumerate}
  \item The structure of graded vector space of $A^{\star}_{\orb}(\IP(w))$ is:
    \begin{align*}
      A_{\orb}^{\star}(\IP(w))=\bigoplus_{g\in\bigcup\limits_{i=0}^n\cU_{w_{i}}}
      A^{\star-a(g)}(|\IP(w_{I(g)})|).
    \end{align*}
  \item The dimension of the vector space
    $A_{\orb}^{\star}(\IP(w))$ is $|w|=w_{0}+\cdots+w_{n}$.
  \item The set $\bs{\eta}:=\left\{\eta_{g}^{d}\mid g\in \bigcup\limits_{i=0}^n\cU_{w_{i}},
    d\in[\![0,\#I(g)-1]\!]\right\}$ is a basis of
    $A_{\orb}^{\star}(\IP(w))$. The orbifold degree of
    $\eta_{g}^{d}$ is $\deg(\eta_g^d)=d+a(g)$.
  \end{enumerate}
\end{proposition}

The second result expresses the orbifold Poincar\'e duality, denoted by
$\langle-,-\rangle$, in the basis $\bs{\eta}$. We set
$<w>:=\prod\limits_{i=0}^{n}w_{i}$.

\begin{proposition}\label{prop:dualite}
   Let ${\eta}^{d_0}_{g_0}$ and
   ${\eta}^{d_1}_{g_1}$ be two elements of the basis
   $\bs{{\eta}}$.
\begin{enumerate}
\item  If  $g_0g_1\neq 1$, then $\langle {\eta}^{d_0}_{g_0},
  {\eta}^{d_1}_{g_1} \rangle = 0$.
   \item If $g_0g_1=1$ then:
  $$
    \langle {\eta}^{d_0}_{g_0},{\eta}^{d_1}_{g_1}\rangle=
  \begin{cases}\frac{1}{<w>} & \text{if } \deg({\eta}^{d_0}_{g_0})+\deg({\eta}^{d_1}_{g_1})=n \\0&
    \text{otherwise.}  \end{cases}
  $$
   \end{enumerate}
 \end{proposition}

 \begin{proof} The vanishings come from the definition of the orbifold Poincar{\'e}
 duality. Assume that $\deg(\eta_{g_0}^{d_0})+\deg(\eta_{g_0^{-1}}^{d_1})=n$.
 According to \cite[Proposition 3.13]{Mann}, one has:
\begin{align*}
    \langle
    {\eta}^{d_0}_{g_0},{\eta}^{d_1}_{g_0^{-1}}\rangle
&=\left(\prod_{i=0}^{n}w_{i}^{-\{\gamma(g_0) w_{i}\}-\{\gamma(g_0^{-1})
    w_{i}\}}\right) \prod_{i\in I(g_0)}w_{i}^{-1}.
\end{align*}
so Formula (\ref{eq:1}) gives $\langle
    {\eta}^{d_0}_{g_0},{\eta}^{d_1}_{g_0^{-1}}\rangle=\frac{1}{<w>}$.
\end{proof}

The third result computes the orbifold cup product, denoted $\cup$, in the
basis $\bs{{\eta}}$.

\begin{proposition}\label{prop:cup} Let
  ${\eta}^{d_{0}}_{g_{0}}$ and
  ${\eta}^{d_{1}}_{g_{1}}$ be two elements of the basis $\bs{{\eta}}$.
  It is:
  \begin{align*}
  {\eta}^{d_{0}}_{g_{0}}\cup {\eta}^{d_{1}}_{g_{1}}=
  {\eta}^{d}_{g_{0}g_{1}}.
  \end{align*}
  with
  $d:=\deg(\eta^{d_{0}}_{g_{0}})+\deg(\eta^{d_{1}}_{g_{1}})-
a(g_{0}g_{1})$.
\end{proposition}

\begin{remark}\label{rem:d,too,big}
   By Proposition \ref{prop:vector} and Formula (\ref{eq:10}), one has:
   $$
   d=a(g_0)+a(g_1)-a(g_0g_1)+d_0+d_1\geq 0.
   $$
   The formula for the cup product
   makes sense only with the following conventions:
   \begin{itemize}
\item If $g_{0}g_{1}\notin\bigcup\limits_{i=0}^n\cU_{w_i}$
   then $\eta_{g_{0}g_{1}}^{d}=0$. The reason is that the component of the inertia
   stack corresponding to $g_0g_1\in\cU$ is empty.
   \item
if $d>\dim\IP(w_{I(g_{0}g_{1})})$ then
   $\eta_{g_{0}g_{1}}^{d}=0$.

 \end{itemize}
\end{remark}

\begin{proof}[Proof of Proposition \ref{prop:cup}]
Set $K(g_{0},g_{1}):= J(g_{0},g_{1})\bigsqcup
\left(I(g_{0}g_{1})\setminus(I(g_{0})\cap I(g_{1}))\right)$. According to
\cite[Corollary 3.18]{Mann}, we have:
\begin{align*}
  {\eta}^{d_{0}}_{g_{0}}\cup {\eta}^{d_{1}}_{g_{1}}=
  \left(\prod_{i=0}^{n}w_{i}^{-\{\gamma(g_{0}w_{i})\}-\{\gamma(g_{1})w_{i}\}+
\{\gamma(g_{0}g_{1})w_{i}\}}\cdot\prod_{i\in K(g_{0},g_{1})} w_{i}
  \right)\cdot{\eta}^{d}_{g_{0}g_{1}}.
  \end{align*}
Formula (\ref{eq:9}) gives the result.
\end{proof}

\section{The model}

\subsection{Construction of the model}\text{}
\label{subsec:model}

Let $w:=(w_0,\ldots,w_n)$ be a sequence of weights. Consider the group
$\cU_{|w|}$ of $|w|$-th roots of the unity and take $\xi:=\exp(2\ii\pi/|w|)$ as
primitive $|w|$-th root. The set
$\bs{\xi}:=\left\{1,\xi,\ldots,\xi^{|w|-1}\right\}$ is a basis of the group
algebra $\IC[\cU_{|w|}]$ and we define:
$$
\deg(\xi^j):=j-|w|\gamma s(j)\quad \forall j=0,\ldots,|w|-1.
$$

\begin{example} Take again $w=(1,2,3)$. The degrees in the group $\cU_6$ are:
$$
\begin{array}{|c||c|c|c|c|c|c|}\hline
\cU_6 & 1 & \xi & \xi^2 & \xi^3 & \xi^4 & \xi^5\\
\hline \deg & 0 & 1 & 2 & 1 & 1 & 1\\
\hline
\end{array}
$$
\end{example}

\begin{lemma}\label{lem:degree}\text{}
\begin{enumerate}
\item For all $j$, $0\leq\deg(\xi^j)\leq n$.

\item For all $j,k$, $\deg(\xi^{j+k})\leq\deg(\xi^j)+\deg(\xi^k)$.
\end{enumerate}
\end{lemma}

\begin{proof}\text{}
\begin{enumerate}
\item By definition, $\kmin(s(j))\leq j\leq\kmax(s(j))$. Formulas
(\ref{eq:5}) and (\ref{eq:5bis}) give:
$$
0\leq a(s(j)^{-1})\leq\deg(\xi^j)\leq n-a(s(j))\leq n.
$$

\item Set $g_0:=s(j)$ and $g_1:=s(k)$. Then:
\begin{align*}
j&=\kmin(g_0)+d_0 \text{ with } d_0\leq \#I(g_0)-1,\\
k&=\kmin(g_1)+d_1 \text{ with } d_1\leq \#I(g_1)-1.
\end{align*}
Using Formulas (\ref{eq:4}) and (\ref{eq:5}) one gets:
\begin{align*}
\deg(\xi^j)&=j-|w|\gamma(g_0)=\kmin(g_0)+d_0-|w|\gamma(g_0)=a(g_0^{-1})+d_0,\\
\deg(\xi^k)&=k-|w|\gamma(g_1)=\kmin(g_1)+d_1-|w|\gamma(g_1)=a(g_1^{-1})+d_1.
\end{align*}
One computes with Formula (\ref{eq:5}):
\begin{align*}
j+k&=\kmin(g_0)+d_0+\kmin(g_1)+d_1\\
&=a(g_0^{-1})+|w|\gamma(g_0)+a(g_1^{-1})+|w|\gamma(g_1)+d_0+d_1\\
&=\deg(\xi^j)+\deg(\xi^k)+|w|\left(\gamma(g_0)+\gamma(g_1)\right).
\end{align*}
Setting $d:=\deg(\xi^j)+\deg(\xi^k)-a((g_0g_1)^{-1})$ and using that:
$$
\gamma(g_0)+\gamma(g_1)\equiv\gamma(g_0g_1) \mod 1,
$$
one gets:
\begin{equation}\label{eq:11}
\xi^j\cdot\xi^k=\xi^{j+k}=\xi^{\kmin(g_0g_1)+d}.
\end{equation}
\begin{itemize}
\item If $\kmin(g_0g_1)+d\leq |w|-1$, then:
\begin{align*}
\deg(\xi^{j+k})&=\kmin(g_0g_1)+d-|w|\gamma s(\kmin(g_0g_1)+d)\\
&= |w|(\gamma(g_0g_1)-\gamma s(\kmin(g_0g_1)+d))+\deg(\xi^j)+\deg(\xi^k).
\end{align*}
By Formula (\ref{eq:6}) one has $\gamma(g_0g_1)\leq\gamma(\kmin(g_0g_1)+d)$,
hence the result.

\item If $\kmin(g_0g_1)+d\geq |w|$ then:
\begin{align*}
\deg(\xi^{j+k})&=\kmin(g_0g_1)+d-|w|-|w|\gamma s(\kmin(g_0g_1)+d-|w|)\\
&= |w|(\gamma(g_0g_1)-1-\gamma s(\kmin(g_0g_1)+d-|w|))+\deg(\xi^j)+\deg(\xi^k)
\end{align*}
and $\gamma(g_0g_1)\leq 1$, hence the result.
\end{itemize}
\end{enumerate}
\end{proof}

\begin{remark}
Looking at (\ref{eq:3}), one observes that if $w_i$ divides $|w|$ for all $i$,
then $|w|\gamma s(j)\in\IN$ for all $j$ so $\deg(\xi^j)$ is an integer for all
$j$. In this case, the orbifold $\IP(w)$ is Gorenstein.
\end{remark}

For any element $z:=\sum\limits_{\sigma
\in\cU_{|w|}}z_{\sigma}\cdot\sigma\in\IC[\cU_{|w|}]$ we set
$\deg({z}):=\max\{\deg{\sigma}\mid z_{\sigma}\neq 0\}$. Introduce the
increasing filtration:
$$
F^u\IC[\cU_{|w|}]:=\lbrace {z}\in\IC[\cU_{|w|}] \mid\deg(z)\leq u\rbrace \text{
for } u\in[0,1[\cap \IQ.
$$
By Lemma \ref{lem:degree}, the natural ring structure on $\IC[\cU_{|w|}]$ is
compatible with this filtration:
\begin{displaymath}
F^u\IC[\cU_{|w|}]\cdot F^v\IC[\cU_{|w|}]\subset F^{u+v}\IC[\cU_{|w|}].
\end{displaymath}
Set $F^{<u}\IC[\cU_{|w|}]:=\lbrace{z}\in\IC[\cU_{|w|}]\mid\deg(z)<u\rbrace$.
The induced product on the graded space
$\gr_F^{\star}\IC[\cU_{|w|}]:=\bigoplus\limits_{u\in[0,1[\cap\IQ}F^u\IC[\cU_{|w|}]/F^{<u}\IC[\cU_{|w|}]$
defines a structure of graded ring denoted $\cup$. For
$\sigma_1,\sigma_2\in\cU_{|w|}$, it is:
\begin{align*}
  \sigma_{1} \cup \sigma_{2}=\left\{\begin{array}{ll} \sigma_{1}\sigma_{2} & \text{if }
\deg(\sigma_{1})+\deg(\sigma_{2})=\deg(\sigma_{1}\sigma_{2})\\0 &
\mbox{otherwise.}\end{array}\right.
\end{align*}

\begin{example} Take again $w=(1,2,3)$. The ring structure on
$\gr_F^\star\IC[\cU_6]$ is given by the table:
\begin{displaymath}
\begin{array}{|c||cccccc|}
\hline
\gr_F^\star\IC[\cU_{6}]& 1 & \xi & \xi^2 & \xi^3 & \xi^4 & \xi^5\\
\hline\hline
1 & 1 & \xi & \xi^2 & \xi^3 & \xi^4 & \xi^5\\
\xi & \xi & \xi^2 & 0 & 0 & 0 & 0 \\
\xi^2 & \xi^2 & 0 & 0 & 0 & 0 & 0 \\
\xi^3 & \xi^3 & 0 & 0 & 0 & 0 & \xi^2  \\
\xi^4 & \xi^4 & 0 & 0 & 0 & \xi^2 & 0 \\
\xi^5 & \xi^5 & 0 & 0 & \xi^2 & 0 & 0 \\
\hline
\end{array}
\end{displaymath}
\end{example}

We define an integral $\int\colon\gr_F^\star\IC[\cU_{|w|}]\to\IC$ by setting
for $j\in[\![0,|w|-1]\!]$:
$$
\int\xi^j=\begin{cases} \frac{1}{<w>} & \text{if } j=n\\0 &
\text{otherwise}\end{cases}
$$
and extending by linearity. The reason is that the only $j$ such that $s(j)=1$
(non twisted sector) and $\deg(\xi^j)=n$ is $j=n$. We further define a pairing
$\model{-}{-}$ on $\gr_F^\star\IC[\cU_{|w|}]$ by setting for
$\sigma_1,\sigma_2\in\cU_{|w|}$:
$$
\model{\sigma_1}{\sigma_2}:=\int\sigma_1\cup\sigma_2
$$
and extending by bilinearity.

\begin{example} Take again $w=(1,2,3)$. The matrix of the pairing $\model{-}{-}$ in the basis
$\bs{\xi}$ is:
\begin{displaymath}
\left( \begin{array}{ccc|ccc}
0&0&1/6&0&0&0\\
0&1/6&0&0&0&0\\
1/6&0&0&0&0&0\\
\hline
0&0&0&0&0&1/6\\
0&0&0&0&1/6&0\\
0&0&0&1/6&0&0\\
  \end{array}\right)
\end{displaymath}
\end{example}

\begin{lemma}
The pairing $\model{-}{-}$ is perfect.
\end{lemma}

\begin{proof}
As in the proof of Lemma \ref{lem:degree}, for $\xi^j=\xi^{\kmin(g_0)+d_0}$
with $g_0=s(j)$ and $d_0\leq\#I(g_0)-1$, set $k:=\kmin(g_0^{-1})+d_1$ with
$d_1:=\#I(g_0)-1-d_0$. Then by Formula (\ref{eq:11}), $\xi^j\cdot\xi^k=\xi^n$
with $\deg(\xi^n)=n$ and by Formula (\ref{eq:2}):
$$
\deg(\xi^j)+\deg(\xi^k)=a(g_0)+a(g_0^{-1})+\#I(g_0)-1=n
$$
so $\model{\xi^j}{\xi^k}=\frac{1}{<w>}$.
\end{proof}

As a consequence, the structure
$\left(\gr^\star_F\IC[\cU_{|w|}],\cup,\model{-}{-}\right)$ is a graded
Frobenius algebra, as
$\left(A^\star_{\orb}(\IP(w)),\cup,\langle-,-\rangle\right)$ is.

\subsection{Isomorphism with the orbifold Chow ring}
\label{subsec:isomorphism}\text{}

Define a linear map $\Xi\colon
A^\star_{\orb}(\IP(w))\to\gr^\star_F\IC[\cU_{|w|}]$ by setting
$\Xi(\eta_g^d):=\xi^{\kmin(g^{-1})+d}$ and extending by linearity.

\begin{example}
Take again $w=(1,2,3)$. The map $\Xi$ is given by:
\begin{align*}
\Xi(\eta_1^0)&=1,&\Xi(\eta_1^1)&=\xi,&\Xi(\eta_1^2)&=\xi^2,&\Xi(\eta_\jj^0)&=\xi^5,
&\Xi(\eta_{-1}^0)&=\xi^4,&\Xi(\eta_{\jj^2}^0)&=\xi^3.
\end{align*}
\end{example}

\begin{theorem}\label{th:isom} The map $\Xi:A^\star_{\orb}(\IP(w))\to\gr^\star_F\IC[\cU_{|w|}]$
is an isomorphism of graded Frobenius algebras.
\end{theorem}

\begin{proof}\text{}
\paragraph{{\bf Step 1:} {\it $\Xi$ is an isomorphim.}} By definition of
$\kmin(g)$ and since $d\leq \#I(g)-1$, as $g$ and $d$ vary, the numbers
$\kmin(g^{-1})+d$ are all distinct and cover $[\![0,|w|-1]\!]$, so the map
$\Xi$ maps the basis $\bs{\eta}$ onto the basis $\bs{\xi}$.

\paragraph{{\bf Step 2:} {\it $\Xi$ is graded.}} It is:
\begin{align*}
\deg(\xi^{\kmin(g^{-1})+d})&=\kmin(g^{-1})+d-|w|\gamma s(\kmin(g^{-1})+d)\\
&=\kmin(g^{-1})+d-|w|\gamma (g^{-1}) \text{ by Formula } (\ref{eq:4})\\
&=a(g)+d \text{ by Formula } (\ref{eq:5})\\
&=\deg(\eta_g^d).
\end{align*}

\paragraph{{\bf Step 3:} {\it $\Xi$ is a ring morphism.}} We use notation of
Proposition \ref{prop:cup} and Step 2. The same computation as in Formula
(\ref{eq:11}) gives:
\begin{equation}\label{eq:12}
\xi^{\kmin(g_0^{-1})+d_0}\cdot\xi^{\kmin(g_1^{-1})+d_1}=\xi^{\kmin((g_0g_1)^{-1})+d}.
\end{equation}
\paragraph{\bf i} Assume that $(g_0g_1)^{-1}\in\bigcup\limits_{i=0}^n\cU_{w_i}$
and $d\leq\dim\IP(w_{I(g_0g_1)})$. Then Formula (\ref{eq:12}) means:
$$
\Xi(\eta_{g_0}^{d_0})\cdot\Xi(\eta_{g_1}^{d_1})=\Xi(\eta_{g_0g_1}^{d})=\Xi(\eta_{g_0}^{d_0}\cup\eta_{g_1}^{d_1}).
$$
Since $\Xi$ is graded, this implies:
$$
\Xi(\eta_{g_0}^{d_0})\cup\Xi(\eta_{g_1}^{d_1})=\Xi(\eta_{g_0}^{d_0}\cup\eta_{g_1}^{d_1}).
$$

\paragraph{\bf ii} Assume that $(g_0g_1)^{-1}\notin\bigcup\limits_{i=0}^n\cU_{w_i}$
or  $d>\dim\IP(w_{I(g_0g_1)})$. Since:
$$
\deg(\xi^{\kmin(g_0^{-1})+d_0})+\deg(\xi^{\kmin(g_1^{-1})+d_1})=\deg(\eta_{g_0}^{d_0})+
\deg(\eta_{g_1}^{d_1})=a(g_0g_1)+d,
$$
we have to show that $\deg(\xi^{\kmin((g_0g_1)^{-1})+d})<a(g_0g_1)+d$.
\begin{itemize}
\item If $\kmin((g_0g_1)^{-1})+d\leq |w|-1$, then:
\begin{align*}
\deg(\xi^{\kmin((g_0g_1)^{-1})+d})&=\kmin((g_0g_1)^{-1})+d-|w|\gamma
s(\kmin((g_0g_1)^{-1})+d)\\
&=a(g_0g_1)+d+|w|\gamma((g_0g_1)^{-1})-|w|\gamma s(\kmin((g_0g_1)^{-1})+d).
\end{align*}
In both cases $d>\#I(g_0g_1)-1$ or
$(g_0g_1)^{-1}\notin\bigcup\limits_{i=0}^n\cU_{w_i}$, one has by Formula
(\ref{eq:6}) $\gamma s(\kmin((g_0g_1)^{-1})+d)>\gamma((g_0g_1)^{-1})$, hence
the result.

\item If $\kmin((g_0g_1)^{-1})+d\geq |w|$, then:
\begin{align*}
\deg(\xi^{\kmin((g_0g_1)^{-1})+d})&=\kmin((g_0g_1)^{-1})+d-|w|\\
&\quad-|w|\gamma s(\kmin((g_0g_1)^{-1})+d-|w|)\\
&=a(g_0g_1)+d+|w|\gamma((g_0g_1)^{-1})-|w|\\
&\quad-|w|\gamma s(\kmin((g_0g_1)^{-1})+d-|w|).
\end{align*}
Since $\gamma((g_0g_1)^{-1})<1+\gamma s(\kmin((g_0g_1)^{-1})+d-|w|)$, one gets
the result.
\end{itemize}

\paragraph{{\bf Step 4:} {\it $\Xi$ is compatible with the pairings.}} We use
notation of Proposition \ref{prop:dualite}.
\begin{itemize}
\item If $g_0g_1\neq 1$, then $s(\kmin((g_0g_1)^{-1})+d)\neq 1$ so using
Formula (\ref{eq:11}) we see that
$\model{\Xi(\eta_{g_0}^{d_0})}{\Xi(\eta_{g_1}^{d_1})}=0=\langle\eta_{g_0}^{d_0},\eta_{g_1}^{d_1}\rangle$.
\item If $g_0g_1=1$, then $d=\deg(\eta_{g_0}^{d_0})+\deg(\eta_{g_1}^{d_1})$ and
$\deg(\xi^{\kmin((g_0g_1)^{-1})+d})=d$ so if $d<n$, then
$\model{\Xi(\eta_{g_0}^{d_0})}{\Xi(\eta_{g_1}^{d_1})}=0=\langle\eta_{g_0}^{d_0},\eta_{g_1}^{d_1}\rangle$
and if $d=n$, then
$\model{\Xi(\eta_{g_0}^{d_0})}{\Xi(\eta_{g_1}^{d_1})}=\frac{1}{<w>}=\langle\eta_{g_0}^{d_0},\eta_{g_1}^{d_1}\rangle$.
\end{itemize}
\end{proof}

\begin{example}
Take again $w=(1,2,3)$. The orbifold Chow ring of $\IP(1,2,3)$ can be pictured
as:
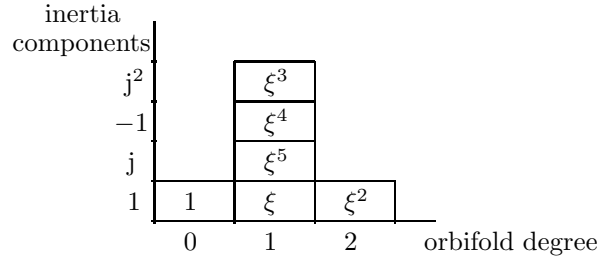
\begin{center}
\begin{picture}(150,115)
\put(44,15){\line(1,0){106}} \put(45,14){\line(0,1){76}}
\put(75,14){\line(0,1){61}}\put(105,14){\line(0,1){61}}\put(135,14){\line(0,1){16}}
\put(44,30){\line(1,0){91}}\put(44,45){\line(1,0){2}}\put(44,60){\line(1,0){2}}
\put(44,75){\line(1,0){2}}
\put(75,45){\line(1,0){30}}\put(75,60){\line(1,0){30}}
\put(75,75){\line(1,0){30}} \put(56,19){\makebox{$1$}}
\put(86,34){\makebox{$\xi^5$}} \put(86,49){\makebox{$\xi^4$}}
\put(86,64){\makebox{$\xi^3$}} \put(86,19){\makebox{$\xi$}}
\put(116,19){\makebox{$\xi^2$}} \put(56,4){\makebox{$0$}}
\put(86,4){\makebox{$1$}}\put(116,4){\makebox{$2$}}
\put(146,4){\makebox{orbifold degree}} \put(35,19){\makebox{$1$}}
\put(35,34){\makebox{$\jj$}}\put(30,49){\makebox{$-1$}}
\put(34,64){\makebox{$\jj^2$}}\put(4,90){\makebox{inertia}}
\put(-8,79){\makebox{components}}
\end{picture}
\end{center}
\end{example}

\begin{example}
Take $w=(1,\ldots,1)$ ($n$ times). Then $\gr^\star_F\IC[\cU_n]\cong\IC[h]/h^n$
where $h$ has degree one. On the other hand, $\IP(w)\cong\IP^n_\IC$ and
$\mathcal{I}\IP(w)=\IP^n_\IC$ so we recover the well-known fact:
$$
A^\star_{\orb}(\IP(1,\ldots,1))\cong\gr^\star_F\IC[\cU_n]\cong\IC[h]/h^n\cong
A^\star(\IP^n_\IC).
$$
\end{example}

\bibliographystyle{amsalpha}
\bibliography{BMPModelBib}

\end{document}